\pdfoutput=1 
\documentclass{ws-rmta}
\usepackage{algorithm}
\usepackage{algpseudocode}
\usepackage{graphicx}

\begin{document}

\markboth{Alexander Dubbs and Alan Edelman}
{The Beta-MANOVA Ensemble with General Covariance}

%
\catchline{}{}{}{}{}
%

\title{The Beta-MANOVA Ensemble with General Covariance}

\author{Alexander Dubbs}

\address{Mathematics, MIT, Massachusetts Ave. and Vassar St.\\
Cambridge, MA 02139, United States of America, \\
\email{alex.dubbs@gmail.com}}

\author{Alan Edelman}

\address{Mathematics, MIT, Massachusetts Ave. and Vassar St.\\
Cambridge, MA 02139, United States of America, \\
\email{edelman@math.mit.edu}}

\maketitle

\begin{history}
\received{\today}
\revised{?}
\end{history}

\begin{abstract}
We find the joint generalized singular value distribution and largest generalized singular value distributions of the $\beta$-MANOVA ensemble with positive diagonal covariance, which is general. This has been done for the continuous $\beta > 0$ case for identity covariance (in eigenvalue form), and by setting the covariance to $I$ in our model we get another version. For the diagonal covariance case, it has only been done for $\beta = 1,2,4$ cases (real, complex, and quaternion matrix entries). This is in a way the first second-order $\beta$-ensemble, since the sampler for the generalized singular values of the $\beta$-MANOVA with diagonal covariance calls the sampler for the eigenvalues of the $\beta$-Wishart with diagonal covariance of Forrester and Dubbs-Edelman-Koev-Venkataramana. We use a conjecture of MacDonald proven by Baker and Forrester concerning an integral of a hypergeometric function and a theorem of Kaneko concerning an integral of Jack polynomials to derive our generalized singular value distributions. In addition we use many identities from Forrester's {\it Log-Gases and Random Matrices}. We supply numerical evidence that our theorems are correct.
\end{abstract}



\section{Introduction}	

The first $\beta$-ensembles were introduced by Dumitriu and Edelman \cite{Dumitriu2002}, the $\beta$-Hermite ensemble and the $\beta$-Laguerre ensemble. A $\beta$-\underline{\indent} ensemble is defined to be a real random matrix with a nonrandom continuous tuning parameter $\beta > 0$ such that when $\beta = 1,2,4$, the $\beta$-\underline{\indent} ensemble has the same joint eigenvalue distribution as the real, complex, or quaternionic \underline{\indent}-ensemble. For $\beta$ not equal to $1,2,4$ its eigenvalue distribution interpolates naturally among the $\beta = 1,2,4$ cases. A $\beta$-circular ensemble and four $\beta$-Jacobi ensembles shortly followed \cite{Lippert2003}, \cite{Killip2004}, \cite{Forrester2005}, \cite{Edelman2008}. The extreme eigenvalues of the $\beta$-Jacobi ensembles were characterized by Dumitriu and Koev in \cite{Dumitriu2008}. More recently, Forrester \cite{Forrester2011} and Dubbs-Edelman-Koev-Venkataramana \cite{Dubbs2013} separately introduced a $\beta$-Wishart ensemble with diagonal covariance, which generalizes the $\beta$-Laguerre ensemble by adding the covariance term.

This paper introduces the $\beta$-MANOVA ensemble with diagonal covariance, which generalizes the $\beta$-Jacobi ensembles by adding the covariance term. When $\beta = 1$ this amounts to finding the distribution of the cosine generalized singular values of the pair $(Y,X\Omega)$, where $X$ is $m\times n$ Gaussian, $Y$ is $p\times n$ Gaussian, and $\Omega$ is $n\times n$ diagonal pds. Note that forcing $\Omega$ to be diagonal does not lose any generality; using orthogonal transformations, were $\Omega$ not diagonal we could replace it with its diagonal matrix of eigenvalues and preserve the model. Our $\beta$-MANOVA ensemble also generalizes the real, complex, and quaternionic MANOVA ensembles (the last of which has never been studied). \cite{Fisher1939}, \cite{Hsu1939}, and \cite{Roy1939} independently solved the $\beta = 1$ identity-covariance case, \cite{Constantine1963} solved our problem in the $\beta = 1$, general-covariance case, and \cite{James1964} solved our problem in the $\beta = 2$, general-covariance case. We find the joint eigenvalue distribution of the $\beta$-MANOVA ensemble, and generalize Dumitriu and Koev's results in \cite{Dumitriu2008} by finding the distribution of the largest generalized singular value of the $\beta$-MANOVA ensemble. We also set the covariance to the identity to add a fourth $\beta$-Jacobi ensemble to the literature in Theorem 3.1. Our $\beta$-MANOVA ensemble is unique in that it is not built on a recursive procedure, rather it is sampled by calling the sampler for the $\beta$-Wishart ensemble.

Generalizations of our results exist in the $\beta = 1,2$ cases by adding a mean matrix to one of the Wishart-distributed parameters. The $\beta = 1$ case is from \cite[p. 1279]{Constantine1963} and the $\beta = 2$ case is from \cite[p. 490]{James1964}.

The sampler for the $\beta$-Wishart ensemble of Forrester \cite{Forrester2011} and Dubbs-Edelman-Koev-Venkataramana \cite{Dubbs2013} is the following algorithm:

\vskip .1 in
\fbox{
\begin{minipage}{4.4 in}
\begin{center}
\underline{Beta-Wishart (Recursive) Model Pseudocode}
\end{center}
\begin{algorithmic}
\State \textbf{Function} $\Sigma$ := BetaWishart$(m,n,\beta,D)$
\If{$n = 1$}
	\State $\Sigma$ := $\chi_{m\beta}D_{1,1}^{1/2}$
\Else
	\State $Z_{1:n-1,1:n-1}$ := BetaWishart$(m,n-1,\beta,D_{1:n-1,1:n-1})$
	\State $Z_{n,1:n-1}$ := $[0,\ldots,0]$
	\State $Z_{1:n-1,n}$ := $[\chi_\beta D_{n,n}^{1/2};\ldots;\chi_{\beta}D_{n,n}^{1/2}]$
	\State $Z_{n,n}$ := $\chi_{(m-n+1)\beta}D_{n,n}^{1/2}$
	\State $\Sigma$ := diag(svd($Z$))
\EndIf
\end{algorithmic}
\end{minipage}
} \vskip .1 in

The elements of $\Sigma$ are distributed according to the following theorem of \cite{Forrester2011} and \cite{Dubbs2013}:

\begin{proposition}
The distribution of the singular values ${\rm diag}(\Sigma) = (\sigma_1,\ldots,\sigma_n)$, $\sigma_1>\sigma_2>\cdots>\sigma_n$, generated by the above algorithm is equal to:
$$ \frac{2^{n}\det(D)^{-m\beta/2}}{\mathcal{K}_{m,n}^{(\beta)}}\prod_{i=1}^{n}\sigma_i^{(m-n+1)\beta-1}\Delta^2(\sigma)^\beta {{{}_0F_0}}^{\!\!(\beta)}\left(-\frac{1}{2}\Sigma^2,D^{-1}\right)d\sigma, $$
where ${{{}_0F_0}}^{\!\!(\beta)}$ and $\mathcal{K}^{(\beta)}_{m,n}$ are defined in the upcoming section, Preliminaries.
\end{proposition}

To get the generalized singular values of the $\beta$-MANOVA ensemble with general covariance, in diagonal $C$, we use the following algorithm which calls ${\rm BetaWishart}(m,n,\beta,D)$. Let $\Omega$ be an $n\times n$ diagonal matrix.

\vskip .1 in
\fbox{
\begin{minipage}{4.4 in}
\begin{center}
\underline{Beta-MANOVA Model Pseudocode}
\end{center}
\begin{algorithmic}
\State \textbf{Function} $C$ := BetaMANOVA$(m,n,p,\beta,\Omega)$
\State $\Lambda := {\rm BetaWishart}(m,n,\beta,\Omega^2)$
\State  $M := {\rm BetaWishart}(p,n,\beta,\Lambda^{-1})^{-1}$
\State $C:= (M + I)^{-\frac{1}{2}}$
\end{algorithmic}
\end{minipage}
} \vskip .1 in

Our main theorem is the joint distribution of the elements of $C$,

\begin{theorem}
The distribution of the generalized singular values ${\rm diag}(C) = (c_1,\ldots,c_n)$, $c_1 > c_2 > \cdots > c_n$, generated by the above algorithm for $m,p \geq n$ is equal to:
\begin{multline*}
\frac{2^n\mathcal{K}_{m+p,n}^{(\beta)}}{\mathcal{K}_{m,n}^{(\beta)}\mathcal{K}_{p,n}^{(\beta)}}\det(\Omega)^{p\beta}\prod_{i=1}^{n}c_i^{(p-n+1)\beta-1}\prod_{i=1}^{n}(1-c_i^2)^{-\frac{p+n-1}{2}\beta-1}\prod_{i<j}|c_i^2-c_j^2|^{\beta} \\ \times{{}_1F_0}^{\!\!(\beta)}\left(\frac{m+p}{2}\cdot\beta;;C^2(C^2-I)^{-1},\Omega^2\right)dc.
\end{multline*}
where ${{{}_1F_0}}^{\!\!(\beta)}$ and $\mathcal{K}^{(\beta)}_{m,n}$ are defined in the upcoming section, Preliminaries.
\end{theorem}

We also find the distributions of the largest generalized singular value in certain cases:
\begin{theorem}
If $t = (m-n+1)\beta/2-1 \in \mathbb{Z}_{\geq 0}$,
\begin{multline}
P(c_1 < x) = \det(x^2\Omega^2((1-x^2)I+x^2\Omega^2)^{-1})^{\frac{p\beta}{2}}\\ \times \sum_{k=0}^{nt}\sum_{\kappa\vdash k,\kappa_1\leq t}\frac{1}{k!}(p\beta/2)_\kappa^{(\beta)}C_\kappa^{(\beta)}\left((1-x^2)((1-x^2)I+x^2\Omega^2)^{-1}\right),
\end{multline}
where the Jack function $C_\kappa^{(\beta)}$ and Pochhammer symbol $(\cdot)_\kappa^{(\beta)}$ are defined in the upcoming section, Preliminaries.
\end{theorem}

These expressions can be computed by Edelman and Koev's software, {\tt mhg}, \cite{Koev2006}.

It is actually intuitive that ${\rm BetaMANOVA}(m,n,p,\beta,\Omega)$ should generalize the real, complex, and quaternionic MANOVA ensembles with diagonal covariance using the ``method of ghosts.'' The method of ghosts was first used implicitly to derive $\beta$-ensembles for the Laguerre and Hermite cases in \cite{Dumitriu2002}, was stated precisely by Edelman in \cite{Edelman2010}, and was expanded on in \cite{Dubbs2013}. To use the method of ghosts, assume a given ensemble is full of $\beta$-dimensional Gaussians, which generalize real, complex, and quaternionic Gaussians and have some of the same properties: they can be left invariant or made into a $\chi_{\beta}$'s under rotation by a real orthogonal or ``ghost orthogonal'' matrix. Then apply enough orthogonal transformations and/or ghost orthogonal transformations to the ghost matrix to make it all real.

In the $\beta$-MANOVA case, let $X$ be $m\times n$ real, complex, quaternion, or ghost normal, $Y$ be $p \times n$ real complex, quaternion, or ghost normal, and let $\Omega$ be $n\times n$ diagonal real pds. Let $\Omega X^*X\Omega$ have eigendecomposition $U\Lambda U^*$, and $\Omega X^*X\Omega(Y^*Y)^{-1}$ have eigendecomposition $VMV^*$. We want to draw $M$ so we can draw $C = {\rm gsvd}_{(\rm cosine)}(Y,X\Omega) = (M+I)^{-\frac{1}{2}}$. Let $\sim$ mean ``having the same eigenvalues.''
$$ \Omega X^*X\Omega(Y^*Y)^{-1} \sim \Lambda U^*(Y^*Y)^{-1}U \sim \Lambda ((U^*Y^*)(YU))^{-1} \sim \Lambda (Y^*Y)^{-1},$$
which we can draw the eigenvalues $M$ of using ${\rm BetaWishart}(p,n,\beta,\Lambda^{-1})^{-1}$. Since $\Lambda$ can be drawn using ${\rm BetaWishart}(m,n,\beta,\Omega^2)$, this completes the algorithm for ${\rm BetaMANOVA}(m,n,p,\beta,\Omega)$ and proves Theorem 1.1 in the $\beta = 1,2,4$ cases.

The following section contains preliminaries to the proofs of Theorems 1.1 and 1.2 in the general $\beta$ case. Most important are several propositions concerning Jack polynomials and Hypergeometric Functions. Proposition 2.1 was conjectured by Macdonald \cite{Macdonald} and proved by Baker and Forrester \cite{Baker1997}, Proposition 2.3 is due to Kaneko, in a paper containing many results on Selberg-type integrals \cite{Kaneko1993}, and the other propositions are found in \cite[pp. 593-596]{Forrester2010}.

\section{Preliminaries}

\begin{definition}
We define the generalized Gamma function to be
$$ \Gamma_n^{(\beta)}(c) = \pi^{n(n-1)\beta/4}\prod_{i=1}^{n}\Gamma(c-(i-1)\beta/2) $$
for $\Re(c) > (n-1)\beta/2$.
\end{definition}

\begin{definition}
$$ \mathcal{K}_{m,n}^{(\beta)} = \frac{2^{mn\beta/2}}{\pi^{n(n-1)\beta/2}}\cdot\frac{\Gamma_n^{(\beta)}(m\beta/2)\Gamma_n^{\beta}(n\beta/2)}{\Gamma(\beta/2)^n}. $$
\end{definition}

\begin{definition}
$$ \Delta(\lambda) = \prod_{i<j}(\lambda_i-\lambda_j). $$
If $X$ is a diagonal matrix,
$$ \Delta(X) = \prod_{i<j}|X_{i,i}-X_{j,j}|. $$
\end{definition}

As in Dumitriu, Edelman, and Shuman, if $\kappa\vdash k$, $\kappa = (\kappa_1,\kappa_2,\ldots,\kappa_n)$ is nonnegative, ordered non-increasingly, and it sums to $k$. Let $\alpha = 2/\beta$. Let $\rho_\kappa^{\alpha} = \sum_{i=1}^{n}\kappa_i(\kappa_i-1-(2/\alpha)(i-1))$. We define $l(\kappa)$ to be the number of nonzero elements of $\kappa$. We say that $\mu \leq \kappa$ in ``lexicographic ordering'' if for the largest integer $j$ such that $\mu_i = \kappa_i$ for all $i < j$, we have $\mu_j \leq \kappa_j$.

\begin{definition}
As in Dumitriu, Edelman and Shuman, \cite{Dumitriu2007} we define the Jack polynomial of a matrix argument, $C^{(\beta)}_\kappa(X)$, as follows: Let $x_1,\ldots,x_n$ be the eigenvalues of $X$. $C^{(\beta)}_\kappa(X)$ is the only homogeneous polynomial eigenfunction of the Laplace-Beltrami-type operator:
$$ D^*_n = \sum_{i=1}^{n}x_i^2\frac{\partial^2}{\partial x_i^2} + \beta\cdot\sum_{1\leq i\neq j \leq n}\frac{x_i^2}{x_i-x_j}\cdot\frac{\partial}{\partial x_i}, $$
with eigenvalue $\rho_k^{\alpha}+k(n-1),$ having highest order monomial basis function in lexicographic ordering (see Dumitriu, Edelman, Shuman, Section 2.4) corresponding to $\kappa$. In addition,
$$ \sum_{\kappa\vdash k,l(\kappa)\leq n} C_{\kappa}^{(\beta)}(X) = {\rm trace}(X)^k.$$
\end{definition}

\begin{definition}
We define the generalized Pochhammer symbol to be, for a partition $\kappa = (\kappa_1,\ldots,\kappa_l)$
$$ (a)_\kappa^{(\beta)} = \prod_{i=1}^{l}\prod_{j=1}^{\kappa_i}\left(a-\frac{i-1}{2}\beta+j-1\right). $$
\end{definition}

\begin{definition}
As in Koev and Edelman \cite{Koev2006}, we define the hypergeometric function ${{}_pF_q}^{\!\!(\beta)}$ to be
$$ {{}_pF_q}^{\!\!(\beta)}(a;b;X,Y) = \sum_{k=0}^{\infty}\sum_{\kappa\vdash k, l(\kappa)\leq n}\frac{(a_1)_\kappa^{\beta}\cdots (a_p)_\kappa^{\beta}}{(b_1)_\kappa^{(\beta)}\cdots (b_q)_\kappa^{(\beta)}}\cdot\frac{C_\kappa^{(\beta)}(X)C_\kappa^{(\beta)}(Y)}{k!C_\kappa^{(\beta)}(I)}. $$
The best software available to compute this function numerically is described in Koev and Edelman, {\tt mhg}, \cite{Koev2006}. ${{}_pF_q}^{\!\!(\beta)}(a;b;X) = {{}_pF_q}^{\!\!(\beta)}(a;b;X,I)$.
\end{definition}

We will also need two theorems from the literature about integrals of Jack polynomials and hypergeometric functions.

Conjectured by MacDonald \cite{Macdonald}, proved by Baker and Forrester \cite{Baker1997} with the wrong constant, correct constant found using {\it Special Functions} \cite[p. 406]{Andrews1999} (Corollary 8.2.2):

\begin{proposition} Let $X$ be a diagonal matrix.
\begin{multline*}
c_n^{(\beta)}\Gamma_n^{(\beta)}(a+(n-1)\beta/2+1)(a+(n-1)\beta/2+1)_\kappa^{(\beta)}C_\kappa^{(\beta)}(Y^{-1}) \\
= |Y|^{a+(n-1)\beta/2+1}\int_{X > 0}{{}_0F_0}^{\!\!(\beta)}(-X,Y)|X|^aC_\kappa^{(\beta)}(X)|\Delta(X)|^\beta dX,
\end{multline*}
where $c_n^{(\beta)} = \pi^{-n(n-1)\beta/4}n!\Gamma(\beta/2)^n\prod_{i=1}^{n}\Gamma(i\beta/2)$.
\end{proposition}

From \cite[p.593]{Forrester2010},
\begin{proposition}
If $X < I$ is diagonal,
$$ {{}_1F_0}^{\!\!(\beta)}(a;;X) = |I-X|^{-a}. $$
\end{proposition}

Kaneko, Corollary 2 \cite{Kaneko1993}:

\begin{proposition} Let $\kappa = (\kappa_1,\ldots,\kappa_n)$ be nonincreasing and $X$ be diagonal. Let $a,b>-1$ and $\beta > 0$.
\begin{multline*}
\int_{0 < X < I}C_\kappa^{(\beta)}(X)\Delta(X)^\beta\prod_{i=1}^{n}\left[x_i^a(1-x_i)^b\right] dX \\
= C_\kappa^{(\beta)}(I)\cdot\prod_{i=1}^{n}\frac{\Gamma(i\beta/2+1)\Gamma(\kappa_i+a+(\beta/2)(n-i)+1)\Gamma(b + (\beta/2)(n-i)+1)}{\Gamma((\beta/2)+1)\Gamma(\kappa_i+a+b+(\beta/2)(2n-i-1)+2)}.
\end{multline*}
\end{proposition}

From \cite[p. 595]{Forrester2010},
\begin{proposition} Let $X$ be diagonal,
\begin{eqnarray*}
{{}_2F_1}^{\!\!(\beta)}(a,b;c;X) &=& {{}_2F_1}^{\!\!(\beta)}(c-a,b;c;-X(I-X)^{-1})|I-X|^{-b} \\
&=& {{}_2F_1}^{\!\!(\beta)}(c-a,c-b;c;X)|I-X|^{c-a-b}. 
\end{eqnarray*}
\end{proposition}

From \cite[p. 596]{Forrester2010},
\begin{proposition}
If $X$ is $n\times n$ diagonal and $a$ or $b$ is a nonpositive integer,
$$ {{}_2F_1}^{\!\!(\beta)}(a,b;c;X) ={{}_2F_1}^{\!\!(\beta)}(a,b;c;I){{}_2F_1}^{\!\!(\beta)}(a,b;a+b+1+(n-1)\beta/2 - c;I-X).   $$
\end{proposition}

From \cite[p. 594]{Forrester2010},
\begin{proposition}
$$ {{}_2F_1}^{\!\!(\beta)}(a,b;c;I) = \frac{\Gamma_n^{(\beta)}(c)\Gamma_n^{(\beta)}(c-a-b)}{\Gamma_n^{(\beta)}(c-a)\Gamma_n^{(\beta)}(c-b)}.  $$
\end{proposition}

\section{Main Theorems}

\noindent\textit{Proof of Theorem 1.1.} Let $m, p \geq n$. We will draw $M$ by drawing $\Lambda \sim P(\Lambda) = {\rm BetaWishart}(m,n,\beta,\Omega^2)$, and compute $M$ by drawing $ M ~ P(M|\Lambda) =  {\rm BetaWishart}(p,n,\beta,\Lambda^{-1})^{-1}$. The distribution of $M$ is $\int P(M|\Lambda)P(\Lambda)d\lambda$. Then we will compute $C$ by $C = (M+I)^{-\frac{1}{2}}$. We use the convention that eigenvalues and generalized singular values are unordered. By the \cite{Dubbs2013} ${\rm BetaWishart}$ described in the introduction, we sample the diagonal $\Lambda$ from

$$ P(\Lambda) = \frac{\det(\Omega)^{-m\beta}}{n!\mathcal{K}_{m,n}^{(\beta)}}\prod_{i=1}^{n}\lambda_i^{\frac{m-n+1}{2}\beta-1}\prod_{i<j}|\lambda_i-\lambda_j|^{\beta}{{}_0F_0}^{(\beta)}\left(-\frac{1}{2}\Lambda,\Omega^{-2}\right)d\lambda, $$
$$ \mathcal{K}_{m,n}^{(\beta)} = \frac{2^{mn\beta/2}}{\pi^{n(n-1)\beta/2}}\cdot\frac{\Gamma_n^{(\beta)}(m\beta/2)\Gamma_n^{\beta}(n\beta/2)}{\Gamma(\beta/2)^n}. $$

Likewise, by inverting the answer to the \cite{Dubbs2013} ${\rm BetaWishart}$ described in the introduction, we can sample diagonal $M$ from

$$ P(M|\Lambda) = \frac{\det(\Lambda)^{p\beta/2}}{n!\mathcal{K}^{(\beta)}_{p,n}}\prod_{i=1}^{n}\mu_i^{-\frac{p-n+1}{2}\beta-1}\prod_{i<j}|\mu_i^{-1}-\mu_j^{-1}|^{\beta}{{}_0F_0}^{(\beta)}\left(-\frac{1}{2}M^{-1},\Lambda\right)d\mu. $$

To get $P(M)$ we need to compute
$$
\frac{\det(\Omega)^{-m\beta}}{n!^2\mathcal{K}_{m,n}^{(\beta)}\mathcal{K}_{p,n}^{(\beta)}}\prod_{i=1}^{n}\mu_i^{-\frac{p-n+1}{2}\beta-1}\prod_{i<j}|\mu_i^{-1}-\mu_j^{-1}|^{\beta}d\mu $$ $$ \times \int_{\lambda_1,\ldots,\lambda_n \geq 0} \det(\Lambda)^{p\beta/2}\prod_{i=1}^{n}\lambda_i^{\frac{m-n+1}{2}\beta-1}\prod_{i<j}|\lambda_i-\lambda_j|^\beta {{}_0F_0}^{(\beta)}\left(\frac{1}{2}\Omega^{-2},-\Lambda\right) {{}_0F_0}^{(\beta)}\left(-\frac{1}{2}M^{-1},\Lambda\right) d\lambda.
$$
Expanding the hypergeometric function, this is
$$
\frac{\det(\Omega)^{-m\beta}}{n!^2\mathcal{K}_{m,n}^{(\beta)}\mathcal{K}_{p,n}^{(\beta)}}\prod_{i=1}^{n}\mu_i^{-\frac{p-n+1}{2}\beta-1}\prod_{i<j}|\mu_i^{-1}-\mu_j^{-1}|^{\beta} \sum_{k=0}^{\infty}\sum_{\kappa\vdash k}\frac{C_\kappa^{(\beta)}\left(-\frac{1}{2}M^{-1}\right)}{k!C_\kappa^{(\beta)}(I)}d\mu$$ $$\times\left[\int_{\lambda_1,\ldots,\lambda_n \geq 0} \prod_{i=1}^{n}\lambda_i^{\frac{m-n+p+1}{2}\beta-1}\prod_{i<j}|\lambda_i-\lambda_j|^\beta {{}_0F_0}^{(\beta)}\left(\frac{1}{2}\Omega^{-2},-\Lambda\right) C_\kappa^{(\beta)}\left(\Lambda\right) d\lambda\right].
$$
Using Proposition 2.1,
$$
\frac{\det(\Omega)^{-m\beta}}{n!^2\mathcal{K}_{m,n}^{(\beta)}\mathcal{K}_{p,n}^{(\beta)}}\prod_{i=1}^{n}\mu_i^{-\frac{p-n+1}{2}\beta-1}\prod_{i<j}|\mu_i^{-1}-\mu_j^{-1}|^{\beta} \sum_{k=0}^{\infty}\sum_{\kappa\vdash k}\frac{C_\kappa^{(\beta)}\left(-\frac{1}{2}M^{-1}\right)}{k!C_\kappa^{(\beta)}(I)}d\mu$$ $$\times\left[\frac{n!\mathcal{K}_{m+p,n}^{(\beta)}}{2^{(m+p)n\beta/2}}\left(\frac{m+p}{2}\beta\right)_\kappa^{(\beta)}\det\left(\frac{1}{2}\Omega^{-2}\right)^{-\frac{m+p}{2}\beta}C_\kappa^{(\beta)}\left(2\Omega^2\right)\right].
$$
Cleaning things up,
\begin{equation*}
\frac{\det(\Omega)^{p\beta}\mathcal{K}_{m+p,n}^{(\beta)}}{n!\mathcal{K}_{m,n}^{(\beta)}\mathcal{K}_{p,n}^{(\beta)}}\prod_{i=1}^{n}\mu_i^{-\frac{p-n+1}{2}\beta-1}\prod_{i<j}|\mu_i^{-1}-\mu_j^{-1}|^{\beta} \sum_{k=0}^{\infty}\sum_{\kappa\vdash k}\frac{C_\kappa^{(\beta)}\left(-\frac{1}{2}M^{-1}\right)}{k!C_\kappa^{(\beta)}(I)}d\mu$$ $$\times\left[\left(\frac{m+p}{2}\cdot\beta\right)_\kappa^{(\beta)}C_\kappa^{(\beta)}\left(2\Omega^2\right)\right].
\end{equation*}
By the definition of the hypergeometric function, this is
\begin{equation}
\frac{\det(\Omega)^{p\beta}\mathcal{K}_{m+p,n}^{(\beta)}}{n!\mathcal{K}_{m,n}^{(\beta)}\mathcal{K}_{p,n}^{(\beta)}}\prod_{i=1}^{n}\mu_i^{-\frac{p-n+1}{2}\beta-1}\prod_{i<j}|\mu_i^{-1}-\mu_j^{-1}|^{\beta}\cdot{{}_1F_0}^{(\beta)}\left(\frac{m+p}{2}\beta;;-M^{-1},\Omega^2\right)d\mu.
\end{equation}
Converting to cosine form, $C = {\rm diag}(c_1,\ldots,c_n) = (M+I)^{-1/2}$, this is
\begin{multline}
\frac{2^n\mathcal{K}_{m+p,n}^{(\beta)}}{n!\mathcal{K}_{m,n}^{(\beta)}\mathcal{K}_{p,n}^{(\beta)}}\det(\Omega)^{p\beta}\prod_{i=1}^{n}c_i^{(p-n+1)\beta-1}\prod_{i=1}^{n}(1-c_i^2)^{-\frac{p+n-1}{2}\beta-1}\prod_{i<j}|c_i^2-c_j^2|^{\beta} \\ \times{{}_1F_0}^{(\beta)}\left(\frac{m+p}{2}\cdot\beta;;C^2(C^2-I)^{-1},\Omega^2\right)dc.
\end{multline}

\begin{theorem} If we set $\Omega = I$ and $u_i = c_i^2$, $(u_1,\ldots,u_n)$ obey the standard $\beta$-Jacobi density of \cite{Lippert2003}, \cite{Killip2004}, \cite{Forrester2005}, and \cite{Edelman2008}.
\begin{equation}
\frac{\mathcal{K}_{m+p,n}^{(\beta)}}{n!\mathcal{K}_{m,n}^{(\beta)}\mathcal{K}_{p,n}^{(\beta)}}\prod_{i=1}^{n}u_i^{\frac{p-n+1}{2}\beta-1}\prod_{i=1}^{n}(1-u_i)^{\frac{m-n+1}{2}\beta-1}\prod_{i<j}|u_i-u_j|^{\beta}du.
\end{equation}
\end{theorem}
\begin{proof}
Proposition 2.2 works from the statement of Theorem 1.1 because $C^2(C^2-I)^{-1} < I$ (we know that $M>0$ from how it is sampled, so $0 < C^2 = (M+I)^{-1} < I$, likewise $C^2-I$).
\begin{multline*}
\frac{2^n\mathcal{K}_{m+p,n}^{(\beta)}}{n!\mathcal{K}_{m,n}^{(\beta)}\mathcal{K}_{p,n}^{(\beta)}}\prod_{i=1}^{n}c_i^{(p-n+1)\beta-1}\prod_{i=1}^{n}(1-c_i^2)^{-\frac{p+n-1}{2}\beta-1}\prod_{i<j}|c_i^2-c_j^2|^{\beta} \\ \times\det(I-C^2(C^2-I)^{-1})^{-\frac{m+p}{2}\beta}dc,
\end{multline*}
or equivalently
$$
\frac{2^n\mathcal{K}_{m+p,n}^{(\beta)}}{n!\mathcal{K}_{m,n}^{(\beta)}\mathcal{K}_{p,n}^{(\beta)}}\prod_{i=1}^{n}c_i^{(p-n+1)\beta-1}\prod_{i=1}^{n}(1-c_i^2)^{\frac{m-n+1}{2}\beta-1}\prod_{i<j}|c_i^2-c_j^2|^{\beta}dc.
$$
If we substitute $u_i = c_i^2$, by the change-of-variables theorem we get the desired result.
\end{proof}

\noindent\textit{Proof of Theorem 1.2.} Let $H = {\rm diag}(\eta_1,\ldots,\eta_n) = M^{-1}$. Changing variables from (3.1) we get
$$ \frac{\det(\Omega)^{p\beta}\mathcal{K}_{m+p,n}^{(\beta)}}{n!\mathcal{K}_{m,n}^{(\beta)}\mathcal{K}_{p,n}^{(\beta)}}\prod_{i=1}^{n}\eta_i^{\frac{p-n+1}{2}\beta-1}\prod_{i<j}|\eta_i-\eta_j|^{\beta}\cdot{{}_1F_0}^{(\beta)}((m+p)\beta/2;;H,-\Omega^2)d\eta. $$
Taking the maximum eigenvalue, following mvs.pdf,
\begin{multline*} P(H < xI) = \frac{\det(\Omega)^{p\beta}\mathcal{K}_{m+p,n}^{(\beta)}}{n!\mathcal{K}_{m,n}^{(\beta)}\mathcal{K}_{p,n}^{(\beta)}} \\ \times\int_{H<xI}\prod_{i=1}^{n}\eta_i^{\frac{p-n+1}{2}\beta-1}\prod_{i<j}|\eta_i-\eta_j|^{\beta}\cdot{{}_1F_0}^{(\beta)}((m+p)\beta/2;;H,-\Omega^2)d\eta, \end{multline*}
Letting $N = {\rm diag}(\nu_1,\ldots,\nu_n) = H/x$, changing variables again we get
\begin{multline*} P(H < xI) = \frac{\det(\Omega)^{p\beta}\mathcal{K}_{m+p,n}^{(\beta)}}{n!\mathcal{K}_{m,n}^{(\beta)}\mathcal{K}_{p,n}^{(\beta)}}\cdot x^{\frac{np\beta}{2}} \\ \times\int_{N<I}\prod_{i=1}^{n}\nu_i^{\frac{p-n+1}{2}\beta-1}\prod_{i<j}|\nu_i-\nu_j|^{\beta}\cdot{{}_1F_0}^{(\beta)}((m+p)\beta/2;;N,-x\Omega^2)d\nu, \end{multline*}
Expanding the hypergeometric function we get
\begin{multline}
P(H < xI) = \frac{\det(\Omega)^{p\beta}\mathcal{K}_{m+p,n}^{(\beta)}}{n!\mathcal{K}_{m,n}^{(\beta)}\mathcal{K}_{p,n}^{(\beta)}}\cdot x^{\frac{np\beta}{2}}\cdot\sum_{k=0}^{\infty}\sum_{\kappa\vdash k}\frac{((m+p)\beta/2)_\kappa^{(\beta)}C_\kappa^{(\beta)}(-x\Omega^2)}{k!C_\kappa^{(\beta)}(I)} \\ \times\left[\int_{N<I}\prod_{i=1}^{n}\nu_i^{\frac{p-n+1}{2}\beta-1}\prod_{i<j}|\nu_i-\nu_j|^{\beta}\cdot C_\kappa^{(\beta)}(N)d\nu\right].
\end{multline}
Using Proposition 2.3,
\begin{multline*}
\int_{N<I}\prod_{i=1}^{n}\nu_i^{\frac{p-n+1}{2}\beta-1}\prod_{i<j}|\nu_i-\nu_j|^{\beta}\cdot C_\kappa^{(\beta)}(N)d\nu \\ =
\frac{C_\kappa^{(\beta)}(I)}{\Gamma(\beta/2+1)^n}\cdot\prod_{i=1}^{n}\frac{\Gamma(i\beta/2+1)\Gamma(\kappa_i+(\beta/2)(p+1-i))\Gamma((\beta/2)(n-i)+1)}{\Gamma(\kappa_i+(\beta/2)(p+n-i)+1)}
\\
= \frac{C_\kappa^{(\beta)}(I)\Gamma_n^{(\beta)} ((n\beta/2)+1)\Gamma_n^{(\beta)}(((n-1)\beta/2)+1)}{\pi^{\frac{n(n-1)\beta}{2}}\Gamma(\beta/2+1)^n}\cdot\prod_{i=1}^{n}\frac{\Gamma(\kappa_i+(\beta/2)(p+1-i))}{\Gamma(\kappa_i+(\beta/2)(p+n-i)+1)}
\end{multline*}
Now
\begin{eqnarray*}
\prod_{i=1}^{n}\Gamma(\kappa_i+(\beta/2)(p+1-i)) &=& \prod_{i=1}^{n}\Gamma((\beta/2)(p+1-i))\prod_{j=1}^{\kappa_i}((\beta/2)(p+1-i)+j-1) \\
&=& \pi^{-\frac{n(n-1)}{4}\beta}\Gamma_n^{(\beta)}(p\beta/2)\prod_{j=1}^{\kappa_i}((\beta/2)(p+1-i)+j-1)\\
&=& \pi^{-\frac{n(n-1)}{4}\beta}\Gamma_n^{(\beta)}(p\beta/2)(p\beta/2)_\kappa^{(\beta)}\\
\prod_{i=1}^{n}\Gamma(\kappa_i+(\beta/2)(p+n-i)+1) &=& \prod_{i=1}^{n}\Gamma((\beta/2)(p+n-i)+1)\prod_{j=1}^{\kappa_i}((\beta/2)(p+n-i)+j) \\
&=& \pi^{-\frac{n(n-1)}{4}\beta}\Gamma_n^{(\beta)}((p+n-1)\beta/2+1)\prod_{j=1}^{\kappa_i}((\beta/2)(p+n-i)+j) \\
&=& \pi^{-\frac{n(n-1)}{4}\beta}\Gamma_n^{(\beta)}((p+n-1)\beta/2+1)((p+n-1)\beta/2+1)_\kappa^{(\beta)}.
\end{eqnarray*}
Therefore,
\begin{multline*}
\int_{N<I}\prod_{i=1}^{n}\nu_i^{\frac{p-n+1}{2}\beta-1}\prod_{i<j}|\nu_i-\nu_j|^{\beta}\cdot C_\kappa^{(\beta)}(N)d\nu \\
= \frac{C_\kappa^{(\beta)}(I)\Gamma_n^{(\beta)} ((n\beta/2)+1)\Gamma_n^{(\beta)}(((n-1)\beta/2)+1)\Gamma_n^{(\beta)}(p\beta/2)}{\pi^{\frac{n(n-1)\beta}{2}}\Gamma(\beta/2+1)^n\Gamma_n^{(\beta)}((p+n-1)\beta/2+1)}\cdot\frac{(p\beta/2)_\kappa^{(\beta)}}{((p+n-1)\beta/2+1)_\kappa^{(\beta)}}
\end{multline*}
Using (3.4) and the definition of the hypergeometric function we get
\begin{multline*}
P(H < xI) = \frac{\det(\Omega)^{p\beta}\mathcal{K}_{m+p,n}^{(\beta)}}{n!\mathcal{K}_{m,n}^{(\beta)}\mathcal{K}_{p,n}^{(\beta)}}\cdot \frac{\Gamma_n^{(\beta)} ((n\beta/2)+1)\Gamma_n^{(\beta)}(((n-1)\beta/2)+1)\Gamma_n^{(\beta)}(p\beta/2)}{\pi^{\frac{n(n-1)\beta}{2}}\Gamma(\beta/2+1)^n\Gamma_n^{(\beta)}((p+n-1)\beta/2+1)} \\ \times x^{\frac{np\beta}{2}}\cdot {{}_2F_1}^{(\beta)}\left( \frac{m+p}{2}\beta,\frac{p}{2}\beta;\frac{p+n-1}{2}\beta+1;-x\Omega^2\right).
\end{multline*}
Rewriting the constant we get
$$
\frac{\Gamma(\beta/2)^n\Gamma_n^{(\beta)}((m+p)\beta/2)}{n!\Gamma_n^{(\beta)}(m\beta/2)\Gamma_n^{(\beta)}(n\beta/2)}\cdot \frac{\Gamma_n^{(\beta)} ((n\beta/2)+1)\Gamma_n^{(\beta)}(((n-1)\beta/2)+1)}{\Gamma(\beta/2+1)^n\Gamma_n^{(\beta)}((p+n-1)\beta/2+1)}.
$$
Commuting some terms gives
$$
\left(\frac{\Gamma(\beta/2)^n\Gamma_n^{(\beta)} ((n\beta/2)+1)}{n!\Gamma(\beta/2+1)^n\Gamma_n^{(\beta)}(n\beta/2)}\right)\cdot \frac{\Gamma_n^{(\beta)}((m+p)\beta/2)\Gamma_n^{(\beta)}(((n-1)\beta/2)+1)}{\Gamma_n^{(\beta)}(m\beta/2)\Gamma_n^{(\beta)}((p+n-1)\beta/2+1)}.
$$
The left fraction in parentheses is
$$ \frac{1}{\prod_{i=1}^{n}(i\beta/2)}\cdot\frac{\prod_{i=1}^{n}\Gamma((n\beta/2)+1 - (i-1)\beta/2)}{\prod_{i=1}^{n}\Gamma((n\beta/2) - (i-1)\beta/2)} = \frac{1}{\prod_{i=1}^{n}(i\beta/2)}\cdot\frac{\prod_{i=1}^{n}\Gamma((i\beta/2)+1)}{\prod_{i=1}^{n}\Gamma(i\beta/2)} = 1.$$
Hence
\begin{multline}
P(H < xI) = \det(\Omega)^{p\beta}\cdot\frac{\Gamma_n^{(\beta)}((m+p)\beta/2)\Gamma_n^{(\beta)}(((n-1)\beta/2)+1)}{\Gamma_n^{(\beta)}(m\beta/2)\Gamma_n^{(\beta)}((n+p-1)\beta/2+1)} \\ \times x^{\frac{np\beta}{2}}\cdot {{}_2F_1}^{(\beta)}\left( \frac{m+p}{2}\beta,\frac{p}{2}\beta;\frac{p+n-1}{2}\beta+1;-x\Omega^2\right).
\end{multline}
Now $H = M^{-1}$ and $C = (M+I)^{-1/2}$, so equivalently,
\begin{multline}
P(C < xI) = \det(\Omega)^{p\beta}\cdot\frac{\Gamma_n^{(\beta)}((m+p)\beta/2)\Gamma_n^{(\beta)}(((n-1)\beta/2)+1)}{\Gamma_n^{(\beta)}(m\beta/2)\Gamma_n^{(\beta)}((n+p-1)\beta/2+1)} \\ \times \left(\frac{x^2}{1-x^2}\right)^{\frac{np\beta}{2}}\cdot {{}_2F_1}^{(\beta)}\left( \frac{m+p}{2}\beta,\frac{p}{2}\beta;\frac{p+n-1}{2}\beta+1;-\frac{x^2}{1-x^2}\cdot\Omega^2\right).
\end{multline}
\noindent\textbf{Remark.} Using $U = {\rm diag}(u_1,\ldots,u_n) = C^2$ and setting $\Omega = I$ this is 
\begin{multline*}
P(U < xI) = \frac{\Gamma_n^{(\beta)}((m+p)\beta/2)\Gamma_n^{(\beta)}(((n-1)\beta/2)+1)}{\Gamma_n^{(\beta)}(m\beta/2)\Gamma_n^{(\beta)}((n+p-1)\beta/2+1)} \\ \times \left(\frac{x}{1-x}\right)^{\frac{np\beta}{2}}\cdot {{}_2F_1}^{(\beta)}\left( \frac{m+p}{2}\beta,\frac{p}{2}\beta;\frac{p+n-1}{2}\beta+1;-\frac{x}{1-x}\cdot I\right),
\end{multline*}
so by using Proposition 2.4, this is
\begin{multline*}
P(U < xI) = \frac{\Gamma_n^{(\beta)}((m+p)\beta/2)\Gamma_n^{(\beta)}(((n-1)\beta/2)+1)}{\Gamma_n^{(\beta)}(m\beta/2)\Gamma_n^{(\beta)}((n+p-1)\beta/2+1)} \\ \times x^{\frac{np\beta}{2}}\cdot {{}_2F_1}^{(\beta)}\left(\frac{n-m-1}{2}\beta+1,\frac{p}{2}\beta;\frac{p+n-1}{2}\beta+1;x I\right),
\end{multline*}
which is familiar from Dumitriu and Koev \cite{Dumitriu2008}.

Now back to the proof of Theorem 1.2. If we use Proposition 2.4 on (3.6) we get
\begin{multline}
P(C < xI) = \det(x^2\Omega^2((1-x^2)I+x^2\Omega^2)^{-1})^{\frac{p\beta}{2}}\cdot\frac{\Gamma_n^{(\beta)}((m+p)\beta/2)\Gamma_n^{(\beta)}(((n-1)\beta/2)+1)}{\Gamma_n^{(\beta)}(m\beta/2)\Gamma_n^{(\beta)}((n+p-1)\beta/2+1)} \\ \times {{}_2F_1}^{(\beta)}\left( \frac{n-m-1}{2}\beta+1,\frac{p}{2}\beta;\frac{p+n-1}{2}\beta+1;x^2\Omega^2((1-x^2)I+x^2\Omega^2)^{-1}\right).
\end{multline}
Using the approach of Dumitriu and Koev \cite{Dumitriu2008}, let $t = (m-n+1)\beta/2-1$ in $\mathbb{Z}_{\geq 0}$. We can prove that the series truncates: Looking at (3.7), the hypergeometric function involves the term
$$ (-t)_\kappa^{(\beta)} = \prod_{i=1}^{n}\prod_{j=1}^{\kappa_i}\left(-t - \frac{i-1}{2}\beta + j-1\right), $$
which is zero when $i = 1$ and $j-1 = t$, so the series truncates when any $\kappa_i$ has $\kappa_i-1\geq t$, or just $\kappa_1 = t+1$. This must happen if $k > nt$. Thus (3.7) is just a finite polynomial,
\begin{multline}
P(C < xI) = \det(x^2\Omega^2((1-x^2)I+x^2\Omega^2)^{-1})^{\frac{p\beta}{2}}\cdot\frac{\Gamma_n^{(\beta)}((m+p)\beta/2)\Gamma_n^{(\beta)}(((n-1)\beta/2)+1)}{\Gamma_n^{(\beta)}(m\beta/2)\Gamma_n^{(\beta)}((n+p-1)\beta/2+1)} \\ \times\sum_{k=1}^{nt}\sum_{\kappa\vdash k, \kappa_1\leq t} \frac{((n-m-1)\beta/2+1)_\kappa^{(\beta)}(p\beta/2)_\kappa^{(\beta)}}{((p+n-1)\beta/2+1)_\kappa^{(\beta)}}\cdot C_\kappa^{(\beta)}(x^2\Omega^2((1-x^2)I+x^2\Omega^2)^{-1}).
\end{multline}
Let $Z$ be a positive-definite diagonal matrix, and $\epsilon$ a real with $|\epsilon| > 0$. Define
$$ f(Z,\epsilon) = \sum_{k=1}^{nt}\sum_{\kappa\vdash k, \kappa_1\leq t} \frac{((n-m-1)\beta/2+1)_\kappa^{(\beta)}(p\beta/2)_\kappa^{(\beta)}}{((p+n-1)\beta/2+1+\epsilon)_\kappa^{(\beta)}}\cdot C_\kappa^{(\beta)}(Z). $$
Using Proposition 2.5,
\begin{multline} f(Z,\epsilon) = {{}_2F_1}^{(\beta)}\left(\frac{n-m-1}{2}\beta+1,\frac{p}{2}\beta;\frac{p+n-1}{2}\beta+1+\epsilon;I\right) \\ \times {{}_2F_1}^{(\beta)}\left(\frac{n-m-1}{2}\beta+1,\frac{p}{2}\beta;\frac{n-m-1}{2}\beta+1-\epsilon;I-Z\right)
\end{multline}
Using the definition of the hypergeometric function and the fact that the series must truncate,
\begin{multline} f(Z,\epsilon) = {{}_2F_1}^{(\beta)}\left(\frac{n-m-1}{2}\beta+1,\frac{p}{2}\beta;\frac{p+n-1}{2}\beta+1+\epsilon;I\right) \\ \times \sum_{k=1}^{nt}\sum_{\kappa\vdash k,\kappa_1\leq t}\frac{ ((n-m-1)\beta/2+1)_\kappa^{(\beta)}(p\beta/2)_\kappa^{(\beta)}}{ ((n-m-1)\beta/2+1-\epsilon)_\kappa^{(\beta)}     }C_\kappa^{(\beta)}(I-Z)
\end{multline}
Now the limit is obvious
\begin{multline} f(Z,0) = {{}_2F_1}^{(\beta)}\left(\frac{n-m-1}{2}\beta+1,\frac{p}{2}\beta;\frac{p+n-1}{2}\beta+1;I\right) \\ \times \sum_{k=1}^{nt}\sum_{\kappa\vdash k,\kappa_1\leq t}(p\beta/2)_\kappa^{(\beta)}C_\kappa^{(\beta)}(I-Z).
\end{multline}
Plugging this expression into (3.8)
\begin{multline*}
P(C < xI) = \det(x^2\Omega^2((1-x^2)I+x^2\Omega^2)^{-1})^{\frac{p\beta}{2}}\cdot\frac{\Gamma_n^{(\beta)}((m+p)\beta/2)\Gamma_n^{(\beta)}(((n-1)\beta/2)+1)}{\Gamma_n^{(\beta)}(m\beta/2)\Gamma_n^{(\beta)}((n+p-1)\beta/2+1)} \\ \times {{}_2F_1}^{(\beta)}\left(\frac{n-m-1}{2}\beta+1,\frac{p}{2}\beta;\frac{p+n-1}{2}\beta+1;I\right) \\ \times \sum_{k=1}^{nt}\sum_{\kappa\vdash k,\kappa_1\leq t}(p\beta/2)_\kappa^{(\beta)}C_\kappa^{(\beta)}((1-x^2)((1-x^2)I+x^2\Omega^2)^{-1}).
\end{multline*}
Cancelling via Proposition 2.6 gives
\begin{multline}
P(C < xI) = \det(x^2\Omega^2((1-x^2)I+x^2\Omega^2)^{-1})^{\frac{p\beta}{2}}\\ \times \sum_{k=0}^{nt}\sum_{\kappa\vdash k,\kappa_1\leq t}\frac{1}{k!}(p\beta/2)_\kappa^{(\beta)}C_\kappa^{(\beta)}\left((1-x^2)((1-x^2)I+x^2\Omega^2)^{-1}\right).
\end{multline}

\section{Numerical Evidence}

The plots below are empirical cdf's of the greatest generalized singular value as sampled by the ${\rm BetaMANOVA}$ pseudocode in the introduction (the blue lines) against the Theorem 1.2 formula for them as calculated by {\tt mhg} (red x's).

\begin{figure}
\centering
\includegraphics[scale=.54]{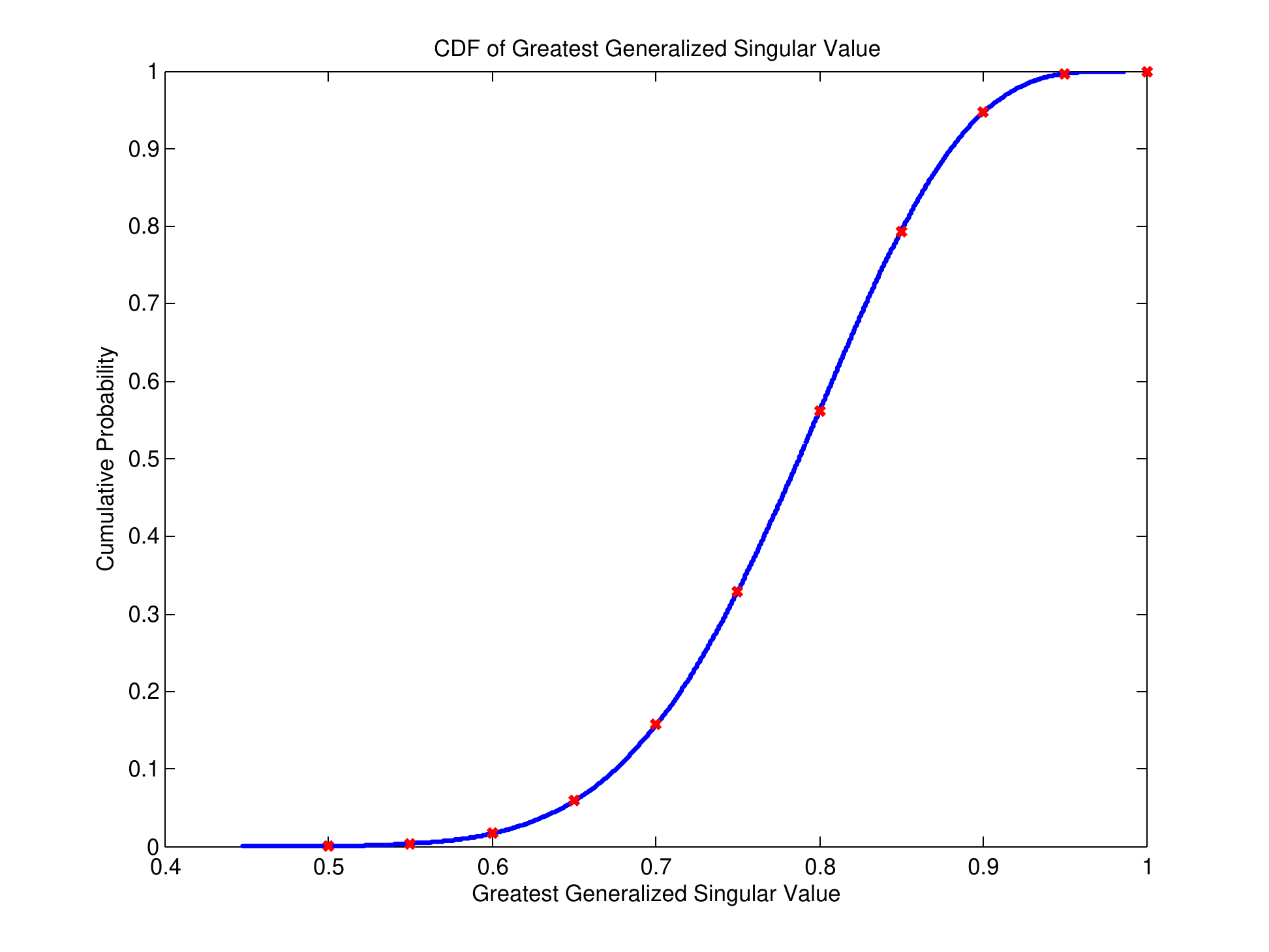}
\caption{Empirical vs. analytic when $m = 7$, $n = 4$, $p = 5$, $\beta = 2.5$, and $\Omega = {\rm diag}(1,2,2.5,2.7)$.}
\end{figure}

\begin{figure}
\centering
\includegraphics[scale=.54]{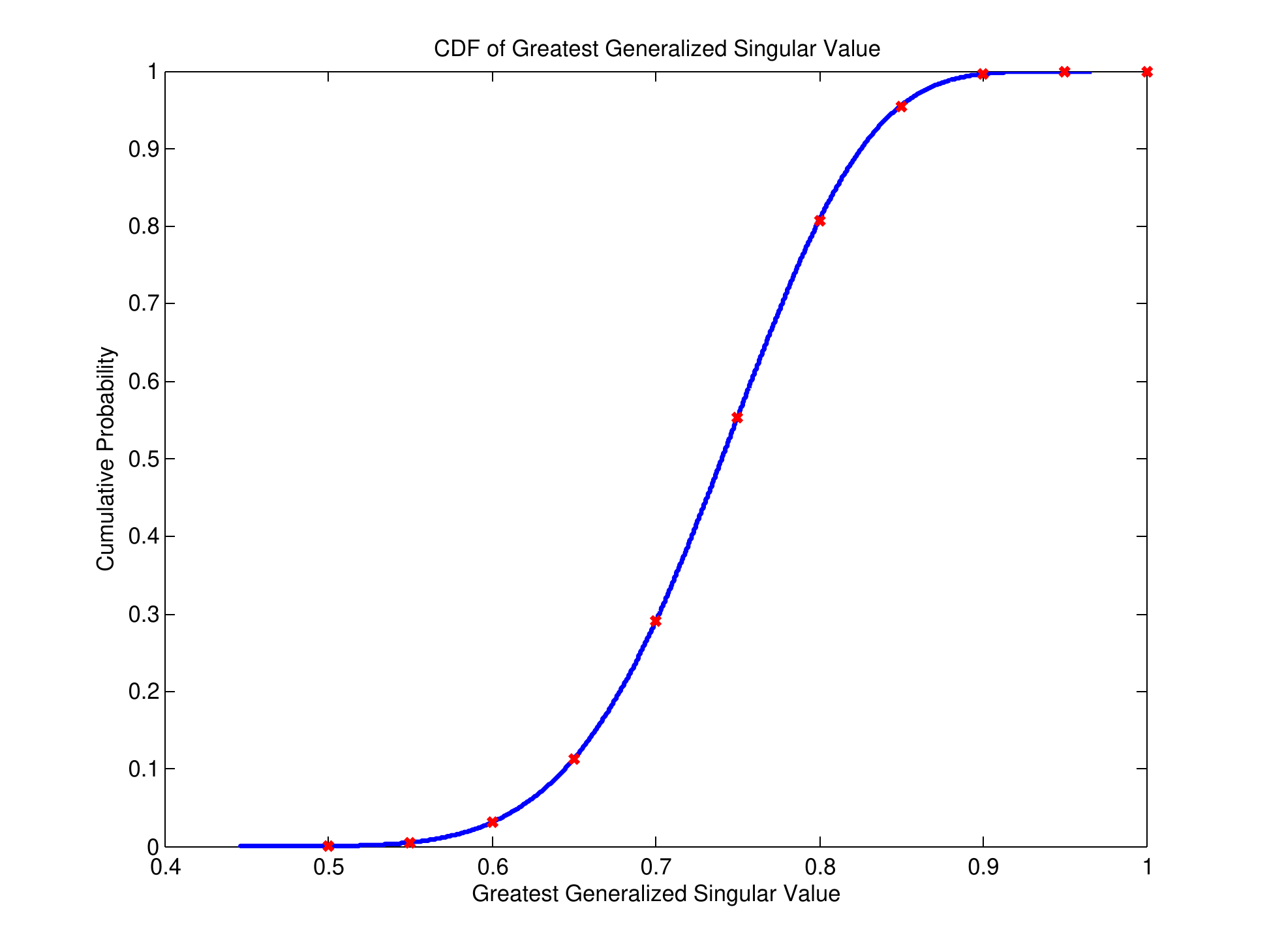}
\caption{Empirical vs. analytic when $m = 9$, $n = 4$, $p = 6$, $\beta = 3$, and $\Omega = {\rm diag}(1,2,2.5,2.7)$.}
\end{figure}

\section*{Acknowledgments}

We acknowledge the support of the National Science Foundation through grants SOLAR Grant No. 1035400, DMS-1035400, and DMS-1016086. Alexander Dubbs was funded by the NSF GRFP.

\end{document}